\documentclass[11pt]{amsart}
\usepackage[utf8]{inputenc}

\usepackage{epsf, amsmath, accents, amsfonts, amscd, amssymb, amsthm, latexsym,verbatim,
dsfont, enumerate,url, graphicx,multicol,times,color,comment}

\def\N{{\mathbb N}}
\def\R{\mathbb{R}}
\def\Z{{\mathbb Z}}

\def\E{{\mathbb E}}
\def\F{{\mathbb F}}

\def\W{{\mathbb W}}

\def\b{\beta}

\newcommand{\eps}{\epsilon}
\def\s{\sigma}

\def\s{\sigma}

\def\cL{{\mathcal B}}
\def\cC{{\mathcal C}}

\def\cL{{\mathcal L}}

\def\cU{{\mathcal U}}

\def\cX{{\mathcal X}}
\def\cY{{\mathcal Y}}

\def\rD{{\rm D}}
\def\rT{{\rm T}}
\def\rd{{\rm d}}

\def\tint{\textstyle\int}

\newcommand{\ti}{\tilde}

\def\la{\langle\,}
\def\ra{\,\rangle}

\newcommand\quotient[2]{
        \mathchoice
            {
                \text{\raise1ex\hbox{$#1$}\Big/\lower1ex\hbox{$#2$}}%
            }
            {
                #1\,/\,#2
            }
            {
                #1\,/\,#2
            }
            {
                #1\,/\,#2
            }
    }
    
\newcommand\quo[2]{
                \text{\raise1ex\hbox{$#1\!$}\big/\lower1ex\hbox{$\!#2$}}
  }

\newcommand\qu[2]{
                \text{\raise.8ex\hbox{$\scriptstyle#1$}/\lower.8ex\hbox{$\scriptstyle#2$}}
  }

\DeclareMathOperator{\pr}{pr}

\DeclareMathOperator{\supp}{supp}

\newcommand{\id}{\operatorname{id}}

\newcommand{\im}{\operatorname{im}}

\newcounter{qcounter}

\newtheorem{theorem}{Theorem}[section]

\newtheorem{lemma}[theorem]{Lemma}
\newtheorem{proposition}[theorem]{Proposition}
\newtheorem{corollary}[theorem]{Corollary}
\newtheorem{definition}[theorem]{Definition}
\newtheorem{remark}[theorem]{Remark}

\title{Counterexamples in Scale Calculus}

\author{Benjamin Filippenko}
\author{Zhengyi Zhou} 
\author{Katrin Wehrheim}
\address{Polyfold Lab, UC Berkeley, Evans Hall, Berkeley CA 94720-3840}

\begin{document}
\maketitle

\begin{abstract}
We construct counterexamples to classical calculus facts such as the Inverse and Implicit Function Theorems in Scale Calculus -- a generalization of Multivariable Calculus to infinite dimensional vector spaces in which the reparameterization maps relevant to Symplectic Geometry are smooth.
Scale Calculus is a cornerstone of Polyfold Theory, which was introduced by Hofer-Wysocki-Zehnder as a broadly applicable tool for regularizing moduli spaces of pseudoholomorphic curves. 
We show how the novel nonlinear scale-Fredholm notion in Polyfold Theory overcomes the lack of Implicit Function Theorems, by formally establishing an often implicitly used fact: 
The differentials of basic germs -- the local models for scale-Fredholm maps -- vary continuously in the space of bounded operators when the base point changes. We moreover demonstrate that this continuity holds only in specific coordinates, by constructing an example of a scale-diffeomorphism and scale-Fredholm map with discontinuous differentials.
This justifies the high technical complexity in the foundations of Polyfold Theory.
\end{abstract}

\section{From Calculus to Scale Calculus}

The Inverse and Implicit Function Theorems are core facts in Calculus for functions of one or several variables (i.e.\ maps $f:\R^m\to \R^n$). They also hold in all previously known contexts -- e.g.\ on Banach spaces\footnote{A Banach space is a vector space with a norm $X\to[0,\infty), x\mapsto \|x\|$ that induces a complete topology.
The spaces $X=\R^n$ with any norm 
are Banach spaces, but the term usually denotes infinite dimensional Banach spaces such as the space of square integrable functions $L^2(\R)=\{f:\R\to\R\,|\, \|f\|_{L^2}:= \int |f(x)|^2|  \rd x <\infty$ \}.
}
and on manifolds\footnote{A manifold is a topological space $X$ that can locally be described in terms of coordinates in $\R^n$. More formally, $X$ is also required to be second countable and Hausdorff, and the charts (local homeomorphisms to $\R^n$) are required to be smoothly compatible, which in particular implies that the dimension $n$ is fixed on connected components of $X$. For an introduction to manifolds see e.g.\ \cite{lee2013smooth}.} -- in which the classical chain rule holds. 

\medskip
\noindent
{\bf Chain Rule:} {\it If two maps $f:X\to Y$ and $g:Y\to Z$ are differentiable, then their composition $g\circ f:X \to Z, x\mapsto g(f(x))$ is differentiable. Its differential at $x\in X$ is given by composition of the differentials, $\rd (g\circ f) (x) = \rd g (f(x)) \circ \rd f(x)$.}
\medskip

Here and throughout we use the Fr\'echet notion of differentiability; see e.g.\ \cite[\S16.2]{lang2013undergraduate}. When $X,Y,Z$ are normed vector spaces, it guarantees that each differential $\rd f (x) : X \to Y$ at a given point $x$ is a linear map. In single variable Calculus for $X=Y=\R$ this map is multiplication $r \mapsto f'(x)r$ by the classical derivative $f'(x)\in\R$. 
More generally, when $X,Y$ are manifolds, then the differential $\rd f (x) : \rT_x X \to \rT_{f(x)} Y$ is a linear map between tangent spaces; see e.g.\ \cite[Prop.3.6]{lee2013smooth}.
The chain rule, e.g.\ \cite[\S15.2]{lang2013undergraduate}, is used for example to prove a formula relating the differentials of a function and its inverse as follows: Since $s^{-1}\circ s = \id $, we have $\rd s^{-1}(s(x)) \circ \rd s(x) = \rd (s^{-1}\circ s) (x) = \rd\,\id (x) = \id$, where $\id$ denotes the identity map on $X$ (and its tangent space $\rT_x X$), and thus the differential of $s^{-1}$ at $s(x)$ is inverse to the differential of $s$ at $x$.
 This is a key ingredient for the following classical result; see e.g.\ \cite[\S17.3]{lang2013undergraduate}.

\medskip
\noindent
{\bf Inverse Function Theorem:} {\it Let $s: X \to Y$ be a continuously differentiable map 
whose differential $\rd s(x_0): X \to Y$ at some $x_0\in X$ is an isomorphism (i.e.\ has a continuous inverse). 
Then there exists a neighborhood $U\subset X$ of $x_0$ such that the map $s:U\to s(U)$ is invertible with open image $s(U)\subset Y$, 
and the inverse $s^{-1}: s(U)\to U$ is continuously differentiable with differential $\rd s^{-1} (s(x))= \rd s(x)^{-1}$.  }
\medskip

Similarly, the chain rule is used to compute the implicit function $y:X\to Y$ that parameterizes  the locus defined by a function $h(x,y)=0$, as follows (for simplicity) in case $X=Y=\R$: Since $h(x,y(x))=0$, we have 
$\partial_x h  + \partial_y h \cdot y'(x)  = 0$ and thus $y'(x) = - \partial_x h / \partial_y h$. 
Note that this requires the partial derivative $\partial_y h$ to be nonzero (or more generally invertible as map $\rT_y Y \to \rT_{h(x,y)} Z$), and this in fact is also a sufficient condition for the local existence of the implicit function $y:X\to Y$, by the following classical result; see e.g.\ \cite[\S17.4]{lang2013undergraduate}.

\medskip
\noindent
{\bf Implicit Function Theorem:} {\it Let $h: X\times Y \to Z$ be a continuously differentiable map whose partial differential $\partial_Y h(x_0,y_0)$ is an isomorphism. Then there exist neighborhoods $U\subset X$ of $x_0$ and $V\subset Y$ of $y_0$ and a differentiable map $y: U \to V$ whose graph parameterizes the local zero set; that is, $h^{-1}(0)\cap (U\times V) = \{ (x, y(x) ) \,|\, x\in U\}$. }
\medskip

This result is critical for Differential Geometry, which studies ``smooth geometric shapes,'' i.e.\ manifolds, by describing them locally in terms of implicit functions. For example, the circle $S^1=\{(x,y)\in \R^2 \,|\, x^2 + y^2 = 1 \}$ has the structure of a 1-dimensional manifold because it can be covered by the four (smoothly compatible) charts arising from applying the Implicit Function Theorem to $h(x,y)=x^2 + y^2 - 1$, 
$$
S^1 \;=\; \{ (x, \pm \sqrt{1-x^2}) \,|\, -1<x<1\} \; \cup\;   \{ (\pm \sqrt{1-y^2}, y) \,|\, -1<y<1\} .
$$
In classical Calculus and Differential Geometry one can also study the zero sets of more general functions such as $h(x,y)=x^2 + y^2$, which do not meet the transversality condition of $\rd h$ being surjective. (Such transversality is equivalent, up to change of coordinates, to a partial differential being an isomorphism.)
{\it Singular zero sets are regularized by perturbing the function to achieve transversality.} The result is a well-defined cobordism class of manifolds of the expected dimension. In our example, $h:\R^2\to \R$ imposes one condition on two variables, so is expected to have 1-dimensional zero set. While the unperturbed zero set $h^{-1}(0)=\{(0,0)\}$ consists of a single point, its perturbations $(h-\eps)^{-1}(0)= \{ (x,y)\in \R^2 \,|\, x^2 + y^2 = \eps \}$ are either circles (for $\eps>0$) or empty sets (for $\eps<0$). These perturbed zero sets are all cobordant. In more elementary terms, the integral of a conservative vector field along $(h-\eps)^{-1}(0)$ is independent of $\eps$ (in fact zero). 

\medskip
{\it Scale Calculus was recently developed by Hofer-Wysocki-Zehnder \cite{HWZbook}\cite{MR2644764} as the cornerstone of Polyfold Theory, which provides an analogous perturbation theory for functions whose zero sets are the moduli spaces of pseudoholomorphic curves studied in Symplectic Geometry.} 
It satisfies a Chain Rule \cite[\S1]{HWZbook}, and with the appropriate scale-Fredholm notion
it satisfies an Implicit Function Theorem. 
But we show in \S\ref{sec:counterex} that, for general scale-differentiable (or even scale-smooth) functions, no reasonable version of the Inverse or Implicit Function Theorems can be true. This does not affect the validity of Polyfold Theory (as we make more explicit in \S\ref{sec:cont}), but it justifies novel extra conditions in the scale-Fredholm notion, and explains the level of technical difficulties in the polyfold approach to overcoming the foundational challenges in regularizing moduli spaces. The latter have been discussed at length, e.g.\ in \cite{fundamentalkuranishi}, and are not the topic of this paper -- apart from one such challenge having motivated the development of Scale Calculus.  
The following remark gives a brief introduction to Scale Calculus from this point of view; for a more in-depth motivation see \cite[\S2.2]{MR3576532}. The Scale Calculus notions require more analysis proficiency than the calculus level discussion so far. The basic claims and constructions in the rest of this paper should be accessible at the advanced calculus level -- when taking for granted the existence of a Scale Calculus in which the constructed maps are ``smooth.'' The proofs are formulated at the undergraduate analysis level such as in \cite{lang2013undergraduate} as much as possible, but require some standard graduate topology and analysis such as compactness considerations and H\"older and Sobolev estimates. Use of prior results in Scale Calculus is labeled.

\begin{remark} \label{rmk:sccalc}\rm
Scale Calculus works with a sequence $\E=(E_i)_{i\in\N_0}$ of Banach spaces with natural embeddings $E_{i+1}\hookrightarrow E_i$. 
This is motivated by the reparameterization map $\tau: \R\times \{f:S^1\to \R\} \to \{f:S^1\to \R\} , (s,f) \mapsto f(s+ \cdot)$ given by viewing the circle as the quotient $S^1=\R/\Z$. Its two-dimensional analogues appear crucially in the description of moduli spaces in symplectic geometry. While $\tau$ is not classically differentiable in any known norm on an infinite dimensional vector space of functions $\{f:S^1\to \R\}$, it is Fr\'echet differentiable as a map $\tau:\cC^{i+1}(S^1)\to \cC^i(S^1)$. Note here the shift in differentiability between the spaces $\cC^i(S^1)=\{ f:S^1\to \R \,|\, f, f', \ldots, f^{(i)} \;\text{continuous} \}$.
This notion of ``shifted differentiability'' reproduces classical Multivariable Calculus by viewing $\R^n$ as the constant sequence $\E=(E_i=\R^n)_{i\in\N_0}$.

Hofer-Wysocki-Zehnder \cite{HWZbook} generalized this notion to infinite dimensions while preserving the chain rule by requiring extra conditions both in the definition of differentiability and on the scale structure $(E_i)_{i\in\N_0}$ as follows: A scale-Banach space is given by sequences of compact\footnote{
Compactness of embeddings means that any bounded sequence in $E_{i+1}$ has a convergent subsequence in $E_i$. When $E_i$ is infinite dimensional, this requires nontrivial embeddings
$E_{i+1}\subsetneq E_i$.
} embeddings $E_{i+1}\hookrightarrow E_i$, whose intersection yields a vector space $E_\infty:=\bigcap_{i\in\N_0} E_i$ that is dense in each $E_i$.
Then a function $\tau:\E\to\F$ is scale-continuous if it is continuous as map $\tau:E_i\to F_i$ for all $i\in\N_0=\{0,1,2,\ldots\}$. Further, the notion of scale-differentiability requires classical differentiability of $\tau:E_{i+1}\to F_i$ together with a well-defined differential $\rd \tau (e) : E_i \to F_i$ for $e\in E_{i+1}$ and continuity of the maps $E_{i+1}\times E_i \to F_i, (e,X)\mapsto \rd\tau(e) X$ for $i\in\N_0$; see \cite[Definitions~1.1, 1.9]{HWZbook}. 
The latter can be phrased as scale-continuity of the tangent map $\rT\tau: \rT \E \to \rT\F, (e,X)\mapsto (\tau(e), \rd\tau(e)X)$, where the shift is encoded in the notion of tangent space $\rT\E=(E_{i+1}\times E_i)_{i\in\N_0}$. 

With these Scale Calculus notions, the above reparameterization map $\tau$ is scale-differentiable and in fact scale-smooth (i.e.\ all its iterated tangent maps $\rT^k\tau$ for $k\in\N$ are scale-differentiable) when specifying $\{f:S^1\to \R\}$ as the scale-Banach space of functions $( \cC^i(S^1) )_{i\in\N_0}$; see \cite[\S2.2]{MR3576532}. 
Here the smooth functions form a dense subspace $E_\infty=\cC^\infty(S^1)$ of each Banach space $E_i= \cC^i(S^1)$ in the scale structure. 
The Banach space $E_0=L^2(\R)$ and scale structure $E_i=H^{i,\delta_i}(\R)$ that we work with in \S\ref{sec:counterex} are somewhat more complicated since we require inner products and wish to work with a space of functions $f:\R\to \R$ whose domain is noncompact.
However, the above example is a good proxy for nonexperts since smooth functions with compact support $\cC^\infty_0(\R)$ are dense in $E_\infty$ and thus in each $E_i$. 
\hfill$\square$
\end{remark}

To regularize moduli spaces of pseudoholomorphic curves despite an absence of Inverse and Implicit Function Theorems, Hofer-Wysocki-Zehnder \cite{MR3683060} show that they are in fact the zero set of scale-Fredholm maps -- a special class of scale-differentiable functions, with the Implicit Function Theorem essentially built into the definition.
This is in stark contrast to classical Fredholm theory -- which establishes, e.g.,\ the Implicit Function Theorem as stated above for continuously differentiable maps between Banach spaces $X\times Y$ and $Z$, when the factor $X$ is finite dimensional. These assumptions are equivalent (after change of coordinates and splitting) to the (generally nonlinear) function $h:X\times Y\to Z$ being transverse (i.e.\ surjective differential) and Fredholm in the classical sense: At every $(x,y)\in h^{-1}(0)$ the differential $\rd h (x,y)$ is a (linear) Fredholm operator; that is, its kernel and cokernel are finite dimensional.
{\it Thus our results demonstrate that the highly nontrivial variation of the nonlinear Fredholm notion in Scale Calculus \cite[Definition 3.7]{HWZbook} is in fact necessary to obtain the desired perturbation theory \cite[Theorems 3.4, 5.5]{HWZbook}.}
This scale-Fredholm notion requires a contraction property -- after change of coordinates and splitting off finite dimensions in domain and target -- and we illuminate this definition in \S\ref{sec:cont} by showing that the contraction property implies a continuity of the differentials. This  is crucial to various proofs of \cite{HWZbook} but only implicitly stated. 
Unfortunately, this continuity holds only in specific coordinates since changes of coordinates in Scale Calculus generally do not preserve continuity of the differential -- another deviation from classical calculus facts that we construct a counterexample for in \S\ref{sec:dis}.
However, our results are sufficient to deduce persistence of transversality in neighborhoods of a transverse zero in Corollary~\ref{cor} for general scale-Fredholm maps. 
This further illuminates why the Inverse and Implicit Function Theorems -- while false for general scale-smooth maps -- actually do hold for scale-Fredholm maps.
\\


\medskip
\noindent
{\bf Acknowledgements:} This work unfolded during lunch at the Mathematical Sciences Research Institute (MSRI), which hosted our Polyfold Lab seminar in Spring 2018. We are profoundly grateful for the inspiring hospitality of the MSRI and the support and curiosity of other seminar participants, particularly Barbara Fantechi, Dusa McDuff, and Helmut Hofer. Invaluable feedback toward making the manuscript widely accessible was provided by Todd Kemp, Dusa McDuff, and the referees.
All authors were supported by NSF grant DMS-1708916.
We dedicate this work to Kris Wysocki.

\section{Counterexamples to Inverse and Implicit Function Theorems} \label{sec:counterex}

A full polyfold analogue of the Inverse Function Theorem would require replacing (open subsets of) Banach spaces by sc-retracts as defined in \cite[Definition 2.2]{HWZbook}. Somewhat simplified, a sc-retract $R=\im\rho$ is the image of a continuous map $\rho:E\to E$ on a Banach space $E$ satisfying $\rho\circ\rho=\rho$, where $E$ is equipped with a scale-structure with respect to which $\rho$ is sc-smooth.\footnote{Here and throughout we usually abbreviate `scale' with `sc.'}
As it turns out, the first nontrivial example of a sc-smooth retraction from \cite[Lemma~1.23]{MR2644764} provides the analytic basis for all of the counterexamples in this section.
To construct it explicitly (and fit our later needs), fix a smooth function $\beta: \R \to [0,\infty)$ with
support in $[-1,1]$ 
and $L^2$-norm $\int_{-\infty}^\infty \beta(x)^2 \rd x = 1$, denote its shifts by $\b_t:= \beta(e^{1/t}+\cdot)$ for $t>0$, and denote the $L^2(\R)$-inner product by $\langle f , g \rangle :=\int_{-\infty}^\infty f(x) g(x) \rd x$.
Then \cite[Lemma~1.23]{MR2644764} gives $E_0=\mathbb{R}\times L^2(\R)$ a scale structure $\E=(E_i)_{i\in\N_0}$ in which the following map is scale-smooth: 
\begin{equation*}
  \rho\,:\;  \mathbb{R}\times L^2(\R) \;\to\; \mathbb{R}\times L^2(\R), \qquad
    \rho (t, f) :=  
    \begin{cases}
    \bigl( t,  \langle f,\beta_t\rangle \, \beta_t  \bigr)&\text{for}\; t>0 ; \\
    (t, 0 ) &\text{for}\; t\leq 0 .
 \end{cases}
\end{equation*}
This map is moreover a retraction in the sense that $\rho\circ\rho = \rho$, and the corresponding sc-retract is 
\begin{equation}\label{eq:R}
R \,:=\; \im\rho \;=\;  \{(t, 0) \,|\, t\leq 0\}\,\cup\, \{(t, s\beta_t) \,|\, t>0, s\in \mathbb{R}\} \;\subset\; \mathbb{R}\times L^2(\R), 
 \end{equation}
with topology induced by its inclusion in $\R\times L^2(\R)$. 
The tangent spaces to this retract are defined as $\rT_{(t,f)} R = \im \rd \rho (t,f)$, which are $1$-dimensional for $t\leq 0$ and $2$-dimensional for $t>0$, as follows for $f=0$ from the computation of the differential
\footnote{
In the case $t=0$ this computation is based on the convergence $\int F \beta_t \to 0$ as $t\searrow 0$ for any fixed $F\in L^2(\R)$.}
\begin{align} \label{drho} 
{\rm d}\rho (t,0) : (T, F) \;\mapsto\; \tfrac{\rd }{\rd \eps}\bigr|_{\eps =0} \rho( t+ \eps T , \eps F \bigr) 
= \begin{cases} 
\bigl( T , \langle F , \beta_t \rangle \beta_t \bigr) &\quad \text{for}\; t>0 ; \\
\bigl( T , 0 \bigr)  &\quad\text{for}\; t\leq 0 . 
\end{cases}
\end{align}
While $\rho$ is not classically differentiable (see Remark~\ref{rmk:differentiability}), the above map is the differential of $\rho$ in scale calculus.
And from here we quickly obtain a first counterexample to the Inverse Function Theorem, in which the map is not invertible since it is not even locally surjective.

\begin{lemma}\label{lem:retract}
There exists a sc-smooth map $s: O \to R$ between sc-retracts $O,R$, whose differential $\rd s(0): \rT_0O \to \rT_{s(0)}R$ is a sc-isomorphism, but $s(O)\subset R$ contains no neighborhood of $s(0)$. 
\end{lemma}

\begin{proof}
The identity map $\R\to\R$ is a sc-smooth retraction with corresponding sc-retract given by $O:=\R$. 
Then the map $s:O\to R, t \mapsto (t,0)$ is sc-smooth (as defined in \cite[Definition~2.4]{HWZbook}) since $\R\to \R\times L^2(\R), t \mapsto (t,0)$ is linear and thus sc-smooth. 
Its differential at $0\in O$ is the map ${\rm d} s(0): T \mapsto (T, 0)$ from $T_0 O = \R$ to $T_{(0,0)} R = \im {\rm d} \rho (0,0) = \R\times \{0\}$ from \eqref{drho}. While this differential is an isomorphism, the image $s(O)\subset R$ does not contain any element of the line $(t , t \beta_t ) \in O$ for $t>0$, which for $t\to 0$ converges to $s(0)=(0,0)$ as $\| t \beta_t\|_{L^2} = t$.
\end{proof}

Next, we show that the complications are not caused by the retracts, but by the differences between classical and scale differentiability.

\begin{lemma} \label{lem:IFT}
There exists a sc-smooth map $s: \E \to \E$ on a sc-Banach space $\E$, whose differential $\rd s(0): \E \to \E$ is a sc-isomorphism, but $s(\E)\subset\E$ contains no neighborhood of $s(0)$. 
\end{lemma}

\begin{proof}
After giving $\E = \R\times L^2(\R)$ a sc-Banach space structure as in \cite[Lemma~1.23]{MR2644764}, we obtain a sc-smooth map
$$
s \,:\; \R\times L^2(\R) \;\to\; \R\times L^2(\R),  \qquad (t, f) \;\mapsto\;  (2t,f) - \rho(t,f) 
= \begin{cases}
\bigl( t , f - \langle f , \beta_t \rangle \beta_t \bigr) &\quad \text{for}\; t>0 ; \\
\bigl( t , f \bigr)  &\quad\text{for}\; t\leq 0 . 
\end{cases}
$$ 
Its differential $\rd s(0,0): (T,F) \mapsto (2T,F) - \rd \rho(0,0)(T,F) = (T,F)$ is the identity, hence an isomorphism, but the image of $s$ does not contain the line $(t , t \beta_t ) \underset{t \to 0}{\to} (0,0)$ for $t>0$ since $f\mapsto  f - \langle f , \beta_t \rangle \beta_t$ is projection to the orthogonal complement of $\R\beta_t\subset L^2(\R)$. 
\end{proof}

In fact, local invertibility is unclear even if the differentials are sc-isomorphisms on an open set. 

\medskip
\noindent
{\bf Question: }{\rm 
Given a sc-smooth map $s: \E \to \F$, whose differential $\rd s(e): \E \to \F$ is a sc-isomorphism for every $e\in \E$, is $s$ (locally) bijective? 
}

\medskip

We suspect that the answer may in fact be `no' as we have the following example with discontinuous inverse.

\begin{lemma} \label{lem:IFT2}
There exists a sc-smooth map $\tilde s: \F \to \F$, whose differential $\rd \tilde s(e): \F \to \F$ is a sc-isomorphism for every $e\in \F$, but whose inverse $\tilde s^{-1}: F_i \to F_0$ is not continuous on any scale $i\in\N_0$. 
\end{lemma}

\begin{proof}
We modify the construction of Lemma~\ref{lem:IFT} by adding a new $\R$ component. The map\\ $\ti s: \R \times \R \times L^2(\R)  \to  \R \times \R \times L^2(\R)$ is defined by
\begin{equation}\label{eqn:map}
\ti s \,:\; (t,y,f)  \;\mapsto\; 
\begin{cases}
\bigl( t , y+\phi(t)\la f, \beta_t\ra,  f - \la f , \beta_t \ra \beta_t + y \phi(t)\beta_t \bigr) &\quad \text{for}\; t>0 ; \\
\bigl( t , y, f \bigr)  &\quad\text{for}\; t\leq 0 . 
\end{cases}
\end{equation}
where $\phi(t) = 0 $ for $t\le 0$ and $\phi(t) =e^{-e^{1/t^2}}$ for $t>0$. We will show that this choice of $\phi\in\cC^\infty(\R)$ guarantees 
classical smoothness of 
\begin{equation}\label{eqn:g_0}
g_0 \,:\; \R \;\to\; H^{k,\delta}(\R), \qquad t \;\mapsto\; \phi(t)\beta_t  \qquad\text{for}\; k\geq 0, \delta\geq 0.
\end{equation}
Here the weighted Sobolev space $H^{k,\delta}(\R)$ is the completion of the smooth compactly supported functions $\cC^\infty_0(\R)$ with respect to the norm
$\| f \|_{H^{k,\delta}} = \sum_{i=0}^k \| e^{\delta |x|} f^{(i)}(x) \|_{L^2}$.  
Smoothness of \eqref{eqn:g_0} then implies classical smoothness of $\R\times L^2(\R) \to \R, \; (t,f) \mapsto \phi(t)\la f, \beta_t\ra = \la f , g_0(t) \ra$ and thus, together with Lemma~\ref{lem:IFT}, proves sc-smoothness of \eqref{eqn:map} -- using a scale structure $E_i=H^{i,\delta_i}(\R)$ for $\delta_{i+1}>\delta_i\geq 0$ on $E_0=L^2(\R)$. 
To show smoothness of \eqref{eqn:g_0} we express it in the general form $g(t) = \psi(t) \phi(t) \gamma_t$ with $\gamma_t:=\gamma (e^{1/t} + \cdot)$ for $\gamma=\beta$ and $\psi\equiv 1$.
Any map of this form with $\psi\in\cC^\infty((0,\infty))$ and compactly supported $\gamma=\tfrac{\rd^k}{\rd t^k}\beta \in\cC^\infty(\R)$ has derivative zero for $t\leq 0$, and for $t>0$ we have
$ \tfrac{\rd}{\rd t} g(t) = \psi'(t) \phi(t) \gamma_t 
+ \psi(t) \phi'(t) \gamma_t 
- \psi(t) \phi(t) \tfrac{1}{t^2} e^{1/t} \gamma' _t$. 
So $\tfrac{\rd}{\rd t} g= g_1 + g_2 + g_3$ is the sum of three functions of the same form, with 
$\psi_1(t)=\psi'(t)$, $\psi_2(t)= \tfrac{2}{t^3} e^{1/t^2} \psi(t)$, $\psi_3(t)=\tfrac{1}{t^2} e^{1/t} \psi(t)$,
$\gamma_1=\gamma_2=\gamma$, and $\gamma_3=\gamma'$. Thus to prove continuity of all derivatives of $g$ it suffices
to prove $\|w_\delta \psi(t) \phi(t) \gamma_t\|_{L^2}= \psi(t) \phi(t) \|w_\delta  \gamma_t\|_{L^2}\to 0$ for $t\searrow 0$ with weight function $w_\delta(x)=e^{\delta |x|}$ and any function $\psi$ 
obtained from $\psi_0(t)=1$ in finitely many steps of multiplying with $\tfrac{2}{t^3} e^{1/t^2}$ or $\tfrac{1}{t^2} e^{1/t}$, 
or taking the $t$-derivative. 
This yields a convex combination of functions of the form $\psi_{\ell,m,n}(t)=\tfrac{1}{t^\ell} e^{m/t^2} e^{n/t}$ for $\ell,m,n\in\N$.
Since $\gamma=\tfrac{\rd^k}{\rd t^k}\beta$ is supported in $[-1,1]$
we can estimate 
$$  \textstyle
\|w_\delta \gamma_t\|_{L^2}^2 \;=\; \int_{-\infty}^\infty \bigl| e^{\delta |y - e^{1/t}|} \gamma(y) \bigr|^2  \rd y \;\leq\; \|\gamma\|_\infty e^{\delta (e^{1/t} + 1)} \;\leq\; C e^{\delta e^{1/t}}.
$$
Then change of variables $x=\frac 1 {t^2} \to \infty$ yields the desired convergence
\begin{align*}
\lim_{t \searrow 0} \psi_{\ell,m,n}(t) \phi(t) \|w_\delta  \gamma_t\|_{L^2}
&\;\leq\; 
\lim_{t \searrow 0} \tfrac{1}{t^\ell} e^{\delta e^{1/t} + m/t^2 + n/t -e^{1/t^2}} \\ &
\;=\; 
\left( \lim_{x\to\infty} x^{\frac \ell 2} e^{- \frac 12 e^x} \right) \cdot e^{
\lim_{x\to\infty} ( \delta e^{\sqrt{x}} + m x + n \sqrt{x} - \frac 12 e^x )}
\;=\; 0 .
\end{align*}
To prove that the differentials $\rd \ti s (t,y,f)$ are sc-isomorphisms for all $(t,y,f)\in\R^2\times L^2(\R)$, first note that the differential is the identity for $t\leq 0$. Next, for fixed $t>0$ and splitting off the first $\R$-factor, the map
$s_t:= {\rm pr}_{\R\times L^2(\R)} \circ \ti s(t,\cdot,\cdot): \R \times L^2(\R) \to \R \times L^2(\R)$ is linear with inverse
$$
s_t^{-1}(y,f)  \; = \; 
\left( \frac{\la f, \beta_t \ra}{\phi(t)} \, , \,  f - \la f , \beta_t \ra \beta_t + \frac{y \phi(t) - \la f, \beta_t \ra } {\phi(t)^2} \beta_t \right).
$$
Now the full differential 
$\rd \tilde s(t,y,f) : (T, Y, F) \mapsto \bigl( T \,,\, T \tfrac{\rd}{\rd t} s_t(y,f) + s_t(Y,F) \bigr) $ for $t>0$
has inverse $(T',Y',F') \mapsto \bigl( T', s_t^{-1} \bigl( (Y',F') - T' \tfrac{\rd}{\rd t} s_t(y,f)  \bigr) \bigr)$.
This shows that in fact $\rd \tilde s(t,y,f)$ is a sc-isomorphism for any fixed $(t,y,f)\in\R^2\times L^2(\R)$, since $\beta_t$ is smooth with compact support, so that the bounded linear operators $\rd \ti s(t,y,f)$ and $\rd \ti s(t,y,f)^{-1}$ on $\R^2\times L^2(\R)$ restrict to bounded linear operators on the scales $\R^2\times H^{i,\delta_i}$.
On the other hand, the inverse of the nonlinear map $\ti s$, 
$$
\ti s^{-1} \,:\; (t,y,f)  \;\mapsto\; 
\begin{cases}
\bigl(t, \frac{\la f, \beta_t \ra}{\phi(t)},  f - \la f , \beta_t \ra \beta_t + \frac{y \phi(t) -\la f, \beta_t \ra } {\phi(t)^2} \beta_t \bigr) &\quad \text{for}\; t>0 ; \\
\bigl(t, y, f \bigr)  &\quad\text{for}\; t\leq 0 
\end{cases}
$$
is not even continuous as a map $\R^2\times H^{i,\delta_i}(\R)\to \R^2\times L^2(\R)$. To see this, pick $f \in H^{i,\delta_i}(\R)$ such that $f(x) = e^{-\delta_i |x|} x^{-2}$ for $|x|>1$. Then the second component of $\ti s^{-1}(t,0,f)$ for $0<t\leq 1$ satisfies an estimate 
$$
{\rm pr}_{\R_y}\bigl(\ti s^{-1}(t, 0, f)\bigr) \;=\;
\tfrac{\la f, \beta_t \ra}{\phi(t)} 
\;\geq\; \tfrac{ e^{-\delta_i (e^{1/t}+1)} (e^{1/t}+1)^{-2} }{\phi(t)} 
\;=\;  \tfrac 14 e^{- 2\delta_i e^{1/t} - 2/t + e^{1/t^2}} \;\underset{t\to0}{\longrightarrow}\; \infty , 
$$ 
so does not extend continuously to ${\rm pr}_{\R_y}\bigl(\ti s^{-1}(0, 0, f)\bigr) = 0$.
\end{proof}

\begin{remark} {\rm
Lemma~\ref{lem:IFT2} also provides a counterexample to the Implicit Function Theorem and its classical consequence that zero sets of smooth Fredholm maps with surjective linearization are smooth manifolds, as follows.

{\it 
Let $\hat {s}:\R\times \R \times L^2(\R) \to L^2(\R)$ denote the projection of \eqref{eqn:map} to $L^2(\R)$. Then $\rd \hat{s}$ is surjective everywhere but the zero set $\hat{s}^{-1}(0)$ is
$$
\bigl\{(t,y,0) \,\big|\, t\le 0, y \in \R \bigr\} \; \cup\;  \bigl\{(t, 0, v\beta_t) \,\big|\, t>0, v\in \R \bigr\}.
$$
This subset of $\R^2\times L^2(\R)$ is not a topological manifold, as it admits no manifold chart at $(0,0,0)$.}

This can be seen by failure of local compactness of $\hat{s}^{-1}(0)$ as follows: Given any $\eps>0$, the intersection $\hat{s}^{-1}(0)\cap B_\epsilon$ with the open $\eps$-ball in $\R^2\times L^2(\R)$ centered at $(0,0,0)$ contains the sequence $e_n=(\frac 1n,0, \frac \eps 2 \beta_{1/n})$ for $n > \frac 2{\sqrt{3}} \eps$, which has no convergent subsequence in $\R^2\times L^2(\R)$ since $\|\beta_{1/n} - \beta_{1/m}\|_{L^2(\R)} = 2$ for $m\gg n$. }
\hfill$\square$
\end{remark}

Next, we obtain an even sharper contrast to the classical Implicit Function Theorem by constructing a nonlinear sc-smooth map with surjective Fredholm linearizations that has a branched 1-dimensional zero set.

\begin{theorem} \label{thm:IFT}
There exists a sc-smooth map $h: \R\times\E \to \E, (t,e) \mapsto h_t(e)$ on a sc-Banach space $\E$, whose partial differential $\rd h_0: \E\to\E$ is a sc-isomorphism, but whose zero set branches at $(0,0)$ in the sense that $h^{-1}(0)=\{(t,0),(t,z(t))\}$ with a sc-smooth function 
$z:\R\to \E$ such that $z(t)=0$ for $t\leq 0$ and $z(t)\neq 0$ for $t>0$.
In fact, $h$ is transverse to $0$ in the sense that $\rd h (t,e)$ is surjective for all $(t,e)\in \R\times E$, 
and $\rd h_t (e)$ is surjective whenever $h_t(e)=0$. 
\end{theorem}

\begin{proof} 
We modify the construction of Lemma~\ref{lem:IFT} by dropping the first component in the codomain and adjusting the second to
\begin{equation}\label{eq:construct}
h \,:\; \R\times L^2(\R) \;\to\; L^2(\R),  \quad (t, f) \;\mapsto\;  
h_t(f) :=
\begin{cases}
 f - \phi(t,\langle f , \beta_t \rangle)\, \beta_t  &\quad \text{for}\; t>0 ; \\
 f   &\quad\text{for}\; t\leq 0  ;
\end{cases}
\end{equation} 
for a smooth function $\phi: [0,\infty)\times\R\to \R, (t,x)\mapsto \phi_t(x)$.
The previous example is reproduced by $\phi(t,x)=x$, but for the present result we multiply this function with a $t$-dependent smooth function to obtain e.g.\ $\phi(t,x)=x (1 - e^{-e^{1/t^2}} + x)$. 
To prove sc-smoothness of $h$ -- using the same sc-structure on $\E= L^2(\R)$ as before -- we may subtract the identity on $L^2(\R)$ and consider the map
$$
\R\times\E \;\to\; \E,\qquad (t,f) \;\mapsto\;
h(t,f) - f \;=\; \psi(t,\la f , \beta_t\ra) \cdot \Phi(t,f) .
$$
Here $\Phi(t,f)\mapsto \la f , \beta_t\ra \; \beta_t$ for $t>0$ extends sc-smoothly to $\Phi(t,f)=0$ for $t\leq 0$ by \cite[Lemma~1.23]{MR2644764}, 
and $\psi:\R^2\to\R$ is some smooth function such as $(t,x)\mapsto1 - e^{-e^{1/t^2}} + x$. So by the product and chain rules in scale calculus \cite[\S1]{HWZbook} it remains to prove sc-smoothness of the function $\Psi:\R\times \E \to \R$ given by $\Psi(t,f)= \langle f , \beta_t \rangle$ for $t>0$ and $\Psi(t,f)=0$ for $t\leq 0$.
For $t\neq 0$ this map is smooth and thus sc-smooth. At $(0,f_0)\in\R\times L^2(\R)$ it is sc$^0$ because both terms in  
$$
\bigl| \Psi(t,f) - \Psi(0,f_0) \bigr| 
\;=\; \bigl| \la f , \beta_t \ra \bigr| 
\;\leq\; \|\beta\|_{\cC^0} \| f - f_0 \|_{L^2} + \bigl|\la f_0 , \beta_t \ra \bigr| 
$$
converge to $0$ as $(t,f)\to (0,f_0)$.
Scale differentiability is only required at $(0,f_0)\in\R\times H^{1,\delta_1}(\R)$ with $\delta_1>0$, where we estimate for $t>0$ 
$$
\bigl|  \Psi(t,f) - \Psi(0,f_0) \bigr|
\;=\; \bigl| \la f , \beta_t \ra \bigr| 
\;\leq\;  \| f\|_{H^{0,\delta_1}} \cdot 
\bigl(\tint e^{- 2\delta_1 |s-e^{1/t}| } \, \beta(s) \, \rd s \bigr)^{\frac 12}
\;\leq\;  C \| f\|_{H^{0,\delta_1}} e^{-\delta_1 e^{1/t} } .
$$
This shows differentiability with trivial differential $D\Psi(0,f_0)\equiv 0$ because $\lim_{t\to 0} t^{-1} e^{-\delta_1 e^{1/t} } =0$. Continuity of the differential then boils down to continuity of $(t,f)\mapsto - t^{-2} e^{1/t} \la f , \beta'(\cdot + e^{1/t}) \ra $ at $t=0$, and further differentiability uses analogous estimates with $\beta$ replaced by its (still smooth and compactly supported) derivatives. The required limits are
$\lim_{t\to 0} t^{-k} e^{\ell/t} e^{-\delta_1 e^{1/t} } =0 $ for $k,\ell\in\N$, which holds since for $x=\frac 1t \to \infty$ we know that $e^{\delta_1 e^x}$ grows faster than $x^k e^{\ell x}$.

This proves sc-smoothness of $h: \R\times \E \to \E$. Next, its partial differentials are $\rd h_t =\id$ for $t\leq 0$ but for $t>0$ we compute
$$
\rd h_t (f)\,:\; F \;\mapsto\; F -  \phi'_t(\langle f , \beta_t \rangle) \langle F , \beta_t \rangle \beta_t .
$$
Whenever $c:= \phi'_t(\langle f , \beta_t \rangle) \neq 1$ this is a sc-isomorphim on $\E$ with inverse $G\mapsto G - \frac{c}{c-1} \langle G , \beta_t \rangle \beta_t$, but for $c=1$ it is the projection to the orthogonal complement of $\R\beta_t$ with 1-dimensional kernel and cokernel. 
To find the zero set, we know $h_t^{-1}(0)=\{0\}$ for $t\leq 0$ and compute for $t>0$ 
$$
h_t(f) = 0 \quad\Leftrightarrow\quad  f = \phi_t(\langle f , \beta_t \rangle) \beta_t 
 \quad\Leftrightarrow\quad  f = x \beta_t , \; x = \phi_t(x)  \quad\Leftrightarrow\quad  f \in \{ 0 , e^{-e^{1/t^2}} \beta_t \} 
$$
since for our specific choice of the function $\phi$ we have
$$
x = x (1 - e^{-e^{1/t^2}} + x)  \quad\Leftrightarrow\quad x=0 \quad\text{or}\quad 1 = 1 - e^{-e^{1/t^2}} + x .
$$
This proves the first part of the theorem with $z(t)= e^{-e^{1/t^2}} \beta_t$ for $t>0$, which extends to a sc-smooth path $z:\R\to\E$ by $z(t)=0$ for $t\leq 0$ by classical smoothness of \eqref{eqn:g_0}.

To check transversality of $h$ and prove the final remark, we compute $\phi'_t(x) =  1 - e^{-e^{1/t^2}} + 2x$ 
so that $\phi'_t(x)=1 \Leftrightarrow x = \frac 12 e^{-e^{1/t^2}}$, 
and thus surjectivity of $\rd h_t(f)$ fails exactly for $t>0$ on the hyperplane $\langle f , \beta_t\rangle = \frac 12 e^{-e^{1/t^2}} = \langle \frac 12 z(t) , \beta_t\rangle$. This is the hyperplane through the mid-point $\frac 12 e^{-e^{1/t^2}} \beta_t$ on the line segment between the two zeros $0, z(t)=e^{-e^{1/t^2}}\beta_t$, and orthogonal to the line $\R\beta_t$ through them, so the hyperplane does not intersect the zero set, as claimed.
Moreover, although the differential $\rd h(t,f)$ is generally defined only at $(t,f)\in\R\times E_1$, our particular choice of function allows us to compute, at any $f\in L^2(\R)$ and obtain a prospective differential $\rd h(t,f): (T,F) \mapsto F$ for $t\leq 0$ and for $t>0$ with 
$x_t:=\langle f, \beta_t \rangle$,
$$
\rd h(t,f) : (T, F) \;\mapsto\; \rd h_t (f) F  - T \bigl( (\partial_t \phi_t)(x_t) \beta_t 
+ \phi'_t(x_t) \langle f , \partial_t\beta_t \rangle \beta_t
+  \phi_t(x_t) \partial_t\beta_t \bigr).
$$
To see that this map $\R\times L^2(\R) \to L^2(\R)$ is surjective, we consider an element $G\in L^2(\R)$ in the orthogonal complement to its image and aim to show that it must be zero. From the established properties of $\rd h_t$, the only case that remains to be considered is $t>0$, $f=x_t\beta_t$, $x_t=\frac 12 e^{-e^{1/t^2}}$, and $G\in\R\beta_t$. In that case we use the identity $2 \langle \beta_t , \partial_t \beta_t \rangle = \partial_t \|\beta_t\|^2 =0$ and compute $(\partial_t \phi_t)(x)= - \tfrac{2}{t^3} e^{1/t^2} e^{-e^{1/t^2}} x$ to obtain 
\begin{align*}
\langle \rd h(t,x_t\beta_t)(1,0) , G \rangle 
&= 
(\partial_t \phi_t)(x_t) \langle \beta_t ,  G \rangle 
+ \phi'_t(x_t)  \langle x_t \beta_t , \partial_t\beta_t \rangle  \langle \beta_t , G \rangle
+  \phi_t(x_t) \langle \partial_t\beta_t ,   G \rangle  \\
&= - \tfrac{2}{t^3} e^{1/t^2-e^{1/t^2}} \cdot \tfrac 12 e^{-e^{1/t^2}}   \langle \beta_t ,  G \rangle 
= -  \tfrac{1}{t^3} e^{1/t^2-2 e^{1/t^2}}  \langle \beta_t ,  G \rangle .
\end{align*}
This implies $\langle \beta_t ,  G \rangle=0$ and thus $G=0$, finishing the proof of surjectivity of $\rd h(t,f)$.
\end{proof}

To better understand the origin of these differences between classical and scale versions of calculus, note that the proofs of the Implicit and Inverse Function Theorems rely on surjectivity (and hence invertibility) of the differential persisting in a neighborhood as follows.

\medskip
\noindent
{\bf Openness of Transversality: } {\it Let $s: E \to F$ be a continuously differentiable Fredholm map between two Banach spaces $E,F$ whose differential $\rd s(0)$ is surjective. Then there is a neighborhood $U\subset E$ of $0$ such that $\rd s(p)$ is surjective for all $p\in U$. }

\medskip
The examples of Lemmas~\ref{lem:IFT}, \ref{lem:IFT2}, and Theorem~\ref{thm:IFT} also disprove the scale calculus version of this classical fact. 
In contrast with Corollary~\ref{cor} this also shows that these examples are not sc-Fredholm in the sense of \cite[Definition 3.7]{HWZbook}.

\begin{remark} \label{rmk:surj} 
{\it Lemma~\ref{lem:IFT} constructs $s$ on $\E=\R\times L^2(\R)$ so that $\rd s(t,f)$ is a sc-isomorphism for $t\leq 0$, but for $t>0$ has $1$-dimensional kernel and cokernel.
Lemma~\ref{lem:IFT2} constructs $\tilde s$ on $\F=\R\times\R\times L^2(\R)$ so that
$\pr_{\E}\circ \tilde s|_\E \simeq s$ on $\E\simeq\R\times\{0\}\times L^2(\R)$ recovers $s$ of Lemma~\ref{lem:IFT}.}

\rm
Indeed, the sc-smooth map $s: \R\times L^2(\R) \to \R\times L^2(\R)$ of Lemma~\ref{lem:IFT} has differential given by $\rd s(t,f)=\id$ for $t\leq 0$ but for $t>0$ we compute
$$
\rd s(t,f) : (T,F) \;\mapsto\;  \bigl( T , F -  \langle F , \beta_t \rangle \beta_t 
+ T \tfrac{ e^{1/t}}{t^2} ( \langle f , \beta'_t \rangle \beta_t + \langle f , \beta_t \rangle \beta'_t )  
\bigr).
$$
For $f=0$ the second component simplifies to the projection $F\mapsto F -  \langle F , \beta_t \rangle \beta_t$ to the orthogonal complement of $\R\beta_t$. Thus, $\rd s(t,0)$ is still a sc-Fredholm operator but has kernel $\{0\}\times\R\beta_t$ and cokernel $\F / \im \rd s(t,0) \simeq \R\beta_t$. For $f\neq 0$ a brief computation shows the same. 
The claim on Lemma~\ref{lem:IFT2} follows by setting $y=0$ in \eqref{eqn:map} and dropping the second component.

{\it Theorem~\ref{thm:IFT}, as established in the proof, constructs $h$ so that the partial differential $\rd h_t (e)$ is a sc-isomorphism whenever $t\leq 0$ or $t>0$ and $\langle e - \frac 12 z(t) , \beta_t \rangle \neq  0$. 
However, for $t>0$ on the hyperplane $\langle e - \frac 12  z(t) , \beta_t \rangle = 0$ through $\frac 12 z(t)$ orthogonal to $\R\beta_t$ the differential $\rd h_t (e)$  has $1$-dimensional kernel and cokernel.}

\rm
In fact, this failure of fiber-wise transversality of the sc-smooth family of maps $h_t:\E\to\E$ along some path $t\mapsto x_t\beta_t$ with limit $0\mapsto 0$ (in our case $x_t\beta_t=\frac 12 z(t)$) is a universal effect for any choice of the function $\phi_t$ in the construction \eqref{eq:construct} with branching zero set.
Indeed, with $\psi_t(x):=x-\phi_t(x)$ we have $h_t^{-1}(0) = \{ x | \psi_t(x) = 0 \} \beta_t$ and transversality fails at $\{ x | \psi_t'(x) = 0 \} \beta_t$. So, by the mean value theorem, there is fiber-restricted transversality failure between any two solutions on the line $\R\beta_t$.

While the total differential $\rd h$ is surjective everywhere in this example, it remains an open question whether there is a scale calculus counterexample to the implicit function theorem in which all partial differentials $\rd h_t$ are surjective in a neighborhood of a branching point.
\hfill$\square$
\end{remark}

Remark~\ref{rmk:surj} shows that Openness of Transversality does not hold for general sc-smooth maps.
On the other hand, Corollary~\ref{cor} below proves Openness of Transversality for sc-Fredholm maps -- based on continuity of the differential as an operator in specific coordinates established in Proposition~\ref{prop:cont-basic}. 
The difference between continuity of the differential required by sc-smoothness and continuity as an operator is illuminated in Remark~\ref{rmk:differentiability}.

\begin{remark} \label{rmk:differentiability} \rm
The map $h: \R \times L^2(\R) \to L^2(\R), (t,f) \mapsto f - \phi_t(\langle f , \beta_t \rangle) \beta_t$ with $\phi_t\equiv 0$ for $t\leq 0$, which appears in all counterexamples above for some choice of $\phi_t$, has a continuous tangent map
\begin{align*}
\rT h \, :\; \R \times L^2(\R) \times \R \times L^2(\R) \;\to\; L^2(\R) \times L^2(\R), \qquad
(t,f, T, F) \;\mapsto\; \bigl( h(t,f) , \rd h (t,f) (T,F) \bigr) 
\end{align*}
but the differential is discontinuous as a map to the Banach space of bounded operators
\begin{align*}
\rd h \, : \; \R \times L^2(\R)  \;\to \;\cL\bigl(\R \times L^2(\R) , L^2(\R) \bigr), \qquad
(t,f) \;\mapsto \;\rd h (t,f) .
\end{align*}
Explicitly, we can see that the difference of differentials,
$$
\rd h(t,0)-\rd h(0, 0) \; : \quad (T,F) \;\mapsto\; 
\begin{cases}
 - \phi_t'(0) \langle F , \beta_t \rangle \beta_t &\quad \text{for}\; t>0 ; \\
\quad 0  &\quad\text{for}\; t\leq 0,
\end{cases}
$$ 
converges pointwise to $(0,0)$ as $t\to 0$ since $\langle F , \beta_t \rangle \to 0$ for any fixed $F\in L^2(\R)$. However, the operator norm in $\cL(\R \times L^2(\R) , L^2(\R))$ is bounded below by $\|\rd h(0,0)-\rd h(t, 0)\| \geq \| \phi_t'(0) \langle \beta_t , \beta_t \rangle \beta_t \|_{L^2} = |\phi_t'(0) |\geq 1$ for every $t>0$ and both $\phi_t(x)=x$ and $\phi_t(x)=e^t x$. Here we used $F=\beta_t$ with $\|\beta_t\|_{L^2}=1$. The higher operator norms in the scale structure (whose specifics we do not discuss) are bounded analogously, 
$\|\rd h(0,0)-\rd h(t, 0)\|_{\cL(\R\times H^{i,\delta},H^{i,\delta})} \geq \|\beta_t\|_{H^{i,\delta}}^{-1} \| \phi_t'(0) \langle \beta_t , \beta_t \rangle \beta_t \|_{H^{i,\delta}} = | \phi_t'(0) | \frac{ \|\beta_t\|_{H^{i,\delta}}}{ \|\beta_t\|_{H^{i,\delta}}} \geq 1$. 
In comparison with Proposition~\ref{prop:cont-basic} this shows that $h$ is not equivalent to a basic germ at $(0,0)$ since otherwise Proposition~\ref{prop:cont-basic} would imply continuity of the differential as operator on level $i\geq 1$ for variations of the base point in $\R\times\{0\}$, which lies in the $\infty$-level of the sc-structure on $\R\times L^2(\R)$.

On the other hand, scale smoothness of $h$ requires continuity of the differential only in $\cL\bigl(\R \times H^{1,\delta}(\R) , L^2(\R) \bigr)$, where the Sobolov space 
$H^{1,\delta}(\R)= \{F:\R\to \R\,|\, e^{\delta |x|} F(x), e^{\delta |x|} F'(x) \in L^2\}$ carries a weight $\delta>0$. In that operator norm we have convergence
$\|\rd h(0,0)-\rd h(t, 0)\| \leq \sup_{\| F \|_{H^{1,\delta}}=1}|\phi_t'(0) | \| e^{\delta |x|} F(x) \|_{L^2} \| e^{-\delta |x|} \beta_t(x)\|_{L^2}  \leq e^{-\delta (e^{1/t}-1)}  |\phi_t'(0) |  \to 0$ as $t\searrow 0$.
\hfill$\square$
\end{remark}

\section{Continuity of differential for basic germs} \label{sec:cont}

The examples in \S\ref{sec:counterex} demonstrate that sc-smoothness and Fredholm linearizations are insufficient for an Inverse or Implicit Function Theorem. Instead, recall from \cite[Definitions~3.4--3.7]{HWZbook} that sc-Fredholm sections in polyfold theory are required to be locally equivalent to a {\em basic germ}. 
Here a section can be thought of (locally, and after a notion of filling) as map $s:\E\to\F$ between sc-Banach spaces, with the admissible changes of coordinates being governed by the bundle structure, which is specified for experts in a footnote. 

\begin{definition} \label{def:basic}
A sc-smooth map $s:\E\to\F$ is {\rm sc-Fredholm} at $e_0\in E_\infty$ if it is regularizing\footnote{The regularizing property requires $s^{-1}(F_i)\subset E_i$ for each $i\in\N$. This plays a minor but still necessary role in the proof of the Implicit Function Theorem of scale calculus.} and 
there is an admissible change of  coordinates that brings $s$ into the form of a basic germ at $0$.\footnote{Admissible changes of coordinates are given by a sc$^+$-section $U\to U\triangleleft\F, e\mapsto (e,s_0(e))$ with $s_0(e_0)=s(e_0)$ on a neighborhood $U\subset\E$ of $e_0$ and a strong bundle isomorphism $U\triangleleft\F \to V\triangleleft (\R^N\times\W), (e,f) \mapsto (\psi(e), \Psi_e f )$ covering a sc-diffeomorphism $\psi:U\to V\subset \R^k\times\W$ with $\psi(e_0)=0$. The result of this change of coordinates applied to a map $s:\E\to\F$ is the map
$f: V \to \R^N \times \W, v \mapsto \Psi_{\psi^{-1}(v)} ( s(\psi^{-1}(v))- s_0(\psi^{-1}(v)) )$.}
Such a {\em basic germ} is a sc-smooth map of the form
\begin{align}\label{eq:f}
f\,: \; \R^k \times \W \supset V \;\to\; \R^N \times \W , \qquad (c,w) \;\mapsto\; \bigl( a(c,w) , w - B(c,w) \bigr) ,
\end{align}
where the sc-smooth map $B:V \to \W$ is a contraction on all levels of $\W$, in the sense that 
for any $i\in\N_0$ and $\epsilon > 0$ there exists $\delta>0$ such that for $c\in\R^k$ and $w_1,w_2\in W_i$ with $|c|,|w_1|_i,|w_2|_i < \delta $ we have the {\it contraction property}
\begin{equation}\label{eqn:contract}
\| B(c,w_1) - B(c,w_2) \|_i \le \epsilon \|w_1-w_2\|_i  . 
\end{equation}
Recall here that the sc-space $\W=(W_i)_{i\in\N_0}$ consists of Banach spaces $W_i$ with norm $\|\cdot\|_i$ and 
compact embeddings $W_i \subset W_j$ for $i>j$ such that $W_\infty:= \bigcap_{i\in\N_0} W_i$ is dense in each $W_i$.
\end{definition}

The purpose of this section is to illuminate this nonlinear sc-Fredholm property by proving a continuity property of the differentials of a basic germ, which is implicit in various proofs of \cite{HWZbook}, and does not hold for general sc-smooth maps, as we show in \S\ref{sec:dis}.
Recall from Remark~\ref{rmk:sccalc} that general sc-smooth maps $s:\E\to\F$ restrict to continuously differentiable maps $E_{i+1}\to F_i$ and the differential is continuous as map $E_{i+1}\times E_i \to F_i, (e,X)\mapsto \rd s(e) X$. This can also be phrased as the differential forming a map $\rd s: E_{i+1} \to \cL(E_i,F_i)$; that is, the differential at any given base point $e\in E_{i+1}$ is an element of the vector space $\cL(E_i,F_i)$, which is defined to consist of bounded (i.e.\ continuous) linear operators such as $\rd s(e) : E_i \to F_i$. However, the differential as map that takes the base point $e$ to the linear operator $\rd s(e)$ may not be continuous in the operator norm on the vector space $\cL(E_i,F_i)$; see \cite[Remark~1.1]{HWZbook}. That is -- as in the previous examples of \S\ref{sec:counterex} by Remark~\ref{rmk:differentiability} -- we cannot generally guarantee 
$\| \rd s(e+h) - \rd s(e) \| = \sup_{\|X\|_{E_i}=1} \| \rd s(e+h)X - \rd s(e)X \|_{F_i} \to 0$ as $\|h\|_{E_{i+1}}\to 0$. 
However, the following proposition establishes this type of continuity at $e=0$ if $s=f$ is a basic germ.

\begin{proposition}\label{prop:cont-basic} 
Let $f$ be a basic germ as in \eqref{eq:f}. Then for every 
$i\in\N$ the differential
$$
\rd f:\R^k \times W_{i+1} \;\to\; \cL(\R^k\times W_i, \R^N \times W_i), \qquad (c,w) \;\mapsto\; \rd f (c,w)
$$ 
is continuous at $(0,0)$ with respect to the operator norm on $\cL(\ldots)$.
In fact, the partial differential in the directions of  $\W$, 
$$
\rd_\W f:\R^k \times W_{i+1} \;\to\; \cL(W_i, \R^N \times W_i), \qquad (c,w) \;\mapsto\; \rd f (c,\cdot)|_w
$$ 
is continuous at $(0,0)$ with respect to the $W_i$-topology on $W_{i+1}$.
\end{proposition}

\begin{proof}
First note that ${\rm pr}_{\R^N}\circ f=a:\R^k \times \W \to \R^N$ is a sc-smooth map with finite dimensional codomain $\F=\R^N$, so for any $i\geq 1$ it restricts to a continuously differentiable map $a:\R^k\times W_i \to \R^N$ by \cite[Proposition~1.7]{HWZbook} and triviality of the sc-structure $F_i=\R^N$ from \cite[p.4]{HWZbook}. Therefore ${\rm pr}_{\R^N}\circ \rd f = \rd a : \R^k \times W_i \to \cL(\R^k\times W_i, \R^N)$ is continuous at $(0,0)$ for $i\ge 1$. 
Now the composition of this map with the inclusion $W_{i+1}\to W_i$ yields continuity of 
${\rm pr}_{\R^N}\circ \rd f = \rd a : \R^k \times W_{i+1} \to \cL(\R^k\times W_i, \R^N)$ at $(0,0)$ for $i\geq 1$ with respect to both the $W_{i+1}$-topology and the $W_i$-topology on $W_{i+1}$.

Next, the linear map $(c,w)\mapsto w$ in the second component of $f$ has differential ${\rm pr}_{\W}$, which restricts to the bounded projections $\R^k\times W_i \to W_i$ and does not vary with the base point. Thus, the crucial step for this proof is to show continuity of $\rd B$ at $(0,0)$. 
Sc-differentiability of $B:\R^k\times\W\to\W$, by \cite[Proposition~1.5]{HWZbook} can be split up into existence of partial differentials 
$\rd_{\R^k} B(c,w) : \R^k \to W_0$ and $\rd_\W B(c,w) : W_0 \to W_0$ for $(c,w)\in \R^k\times W_1$, which for $w\in W_{i+1}$ restrict to bounded operators in $\cL(\R^k,W_i)$ resp.\ $\cL(W_i,W_i)$, such that the shifted difference quotients converge,
$$
\lim_{\|(d,h)\|_{\R^k\times W_{i+1}} \to 0} \frac{\| B(c+d,w+h) - B(c, w) - \rd_{\R^k} B(c,w) d - \rd_\W B(c,w) h \|_{W_i}}{\|(d,h)\|_{\R^k\times W_{i+1}}} \;=\; 0 ,
$$
and $(c,w,d) \mapsto \rd_{\R^k} B(c,w) d$ restricts to continuous maps $\R^k\times W_{i+1}\times \R^k \to W_i$, 
as well as $(c,w,h) \mapsto \rd_\W B(c,w) h$ restricts to continuous maps $\R^k\times W_{i+1}\times W_i \to W_i$ for every $i\geq 0$.
For the first component of the differential, $\rd_{\R^k} B$, the vector-wise continuity implies continuity of 
$(c,w) \to \rd_{\R^k} B (c,w)$ in the operator topology $\R^k\times W_{i+1}\to \cL( \R^k ,W_i)$ since the domain $\R^k$ of the bounded operators is finite dimensional.
To show the continuity of $(c,w) \to \rd_{\W} B (c,w)$ in the operator topology $\R^k\times W_{i+1}\to \cL(W_i,W_i)$ at $(0,0)$, recall that, given $\epsilon > 0$, the contraction property \eqref{eqn:contract} provides $\delta>0$ so that $\|B(c,w_1) - B(c,w_2)\|_i < \epsilon \|w_1-w_2\|_i$ whenever $|c|,\|w_1\|_i, \|w_2\|_i < \delta$. We claim that this implies $\|\rd_\W B(c,w)\|_{\cL(W_i,W_i)}\le 2\epsilon$ for $w\in \W_\infty$ with $\|w\|_i<\delta$. 
Indeed, assume by contradiction $\|\rd_\W B(c,w) h\|_i > 2 \epsilon \|h\|_i $ for some $h\in W_i$. 
Since $W_{i+1}$ is dense in $W_i$ and $\rd_\W B(c,w)$ is continuous, we can find a nearby $h\in W_{i+1}$ that satisfies the same inequality.
Then for $t>0$ sufficiently small such that $\|w+th\|_i, \|w\|_i <\delta$ we can bound the shifted difference quotient
\begin{align*}
 \frac{\| B(c,w+th) - B(c,w) - \rd_\W B(c,w) th \|_i}{\|th\|_{i+1}} 
&\ge 
\frac{ t \| \rd_\W B(c,w) h \|_i - \| B(c,w+th) - B(c,w) \|_i }{\|th\|_{i+1}} \\
&\ge 
\frac{2 t \epsilon \|h\|_i - \epsilon \| w+th - w \|_i }{\|th\|_{i+1}} 
\;=\;
\frac{\epsilon \|h\|_i}{\|h\|_{i+1}} \;>\; 0 .
\end{align*}
This contradicts the above condition of sc-differentiability for $d=0$ and $t\to 0$.
Thus, given any $\epsilon > 0$ we found $\delta>0$ so that $\|\rd_\W B(c,w)\|_{\cL(W_i,W_i)}\le 2\epsilon$ for $w\in \W_{i+1}$ with $\|w\|_i<\delta$.
Therefore $\rd_\W B$ is continuous at $(0,0)$ not just in the natural topology on $W_{i+1}$ but even in the coarser topology induced by the embedding $W_{i+1}\subset W_i$. The same is true for ${\rm pr}_{\R^N}\circ \rd f$ with $i\geq1$, which proves the claimed continuity of $\rd_\W f$. 
For $\rd_{\R^k} B$, the scale differentiability yields continuity only in the topology of $W_{i+1}$, so the overall differential $df$ is continuous at $(0,0)$ in the $W_{i+1}$-topology.
\end{proof}

Unfortunately, Proposition~\ref{prop:cont-basic} does not prove continuity of the differential as operator for general sc-Fredholm maps, since a change of coordinates by a nonlinear sc-diffeomorphism of the domain does not generally preserve continuity of the differential, as shown in \S\ref{sec:dis}.
In applications, we do expect sc-Fredholm maps such as the Cauchy-Riemann operator in \cite{MR3683060} to have continuous differentials, as the changes of coordinates in practice are linear -- arising from splitting off kernel and cokernel of linearized operators.
However, we deduce from Proposition~\ref{prop:cont-basic} that any property which (i) follows from continuity of the differential in the operator norm, and (ii) is preserved under admissible changes of coordinates, must also hold for sc-Fredholm maps.
This proves the following scale calculus analogues of ``Openness of Transversality" and ``Openness of isomorphic differentials.'' Here we also note the full polyfold theoretic version of this result in the language of \cite{HWZbook}. 

\begin{corollary} \label{cor}
Let $s:\E\to \F$ be sc-Fredholm in the sense of Definition~\ref{def:basic} at every $e_0\in E_\infty$. Then for any $i\in\N$ the following subsets of $E_\infty$ are open with respect to the $E_{i+1}$-topology,
$$
\{e\in E_\infty \,|\, \rd s(e)(E_i)=F_i\}, \qquad\qquad
\{e\in E_\infty \,|\, \rd s(e):E_i \to F_i \text{ is a sc-isomorphism}\}.
$$
Let $\s:\cX\to \cY$ be a sc-Fredholm section of a strong bundle $P:\cY\to\cX$. Then, given any local trivialization $P^{-1}(\cU) \simeq K \subset \E\triangleleft\F$ over an open subset $\cU\subset\cX$, the following subsets of   $\cU_\infty=\cU\cap \cX_\infty$ are open with respect to the $\cX_{i+1}$-topology for any $i\in\N$,
$$
\{x\in \cU_\infty \,|\, \rD\s(x)(\rT_x\cX_i)=(\cY_x)_i\} , \qquad
\{x\in \cU_\infty \,|\, \rD\s(x):\rT_x\cX_i \to(\cY_x)_i \text{ is a sc-isomorphism}\} .
$$
Here $(\cY_x)_i$ is the $i$-th scale of the fiber $\cY_x:=P^{-1}(x)$, and the linearizations $\rD\s(x)$ are determined by the choice of local trivialization. 
\end{corollary}

\section{Discontinuity of differential for sc-diffeomorphisms} \label{sec:dis}

The purpose of this section is to show that sc-diffeomorphisms  -- in contrast to the basic germs in \S\ref{sec:cont} -- can have discontinuous differential, viewed as a map to the space of bounded linear operators as in Proposition~\ref{prop:cont-basic}.

\begin{theorem}\label{thm:discontinuous}
There exists a sc-diffeomorphism $s: \F \to \F$ on a sc-Banach space $\F=(F_i)_{i\in\N_0}$, whose differential $\rd s: F_{i+1} \to \cL(F_i, F_i)$ is discontinuous for any scale $i\in\N_0$.
\end{theorem}

The map $s:\F\to\F$ in Theorem~\ref{thm:discontinuous} is also an example of a sc-Fredholm map with discontinuous differential, since $s$ is equivalent, via the sc-diffeomorphism $s$, to the identity map $\id_\F$, which is a basic germ (as it satisfies Definition~\ref{def:basic} with $\W=\F$, $k=N=0$, and $B\equiv 0$).

\begin{remark} \label{rmk:scdiffdfn}
A sc-diffeomorphism is defined \cite[p.12]{MR2341834} to be a homeomorphism $f:U\to V$ between open subsets $U\subset\E, V\subset \F$ of sc-Banach spaces, such that both $f$ and $f^{-1}$ are sc-smooth. 
It then follows that the differential
$\rd_u f := \rd f (u) : E_k \to F_k$ is an isomorphism on scale $k\in\N_0$ at base points $u\in U\cap E_{k+1}$. In particular, $\rd f (u): \E \to \F$ is a sc-isomorphism for $u\in U\cap E_\infty$. 

Indeed, the chain rule \cite[Theorem~1.1]{HWZbook} applied to the identities
$g \circ f = \id_U$ and $f \circ g = \id_V$ for $g:=f^{-1}$
yields
$\rd_{f(u)} g \circ \rd_u f = \id_{E_k}$ for $u\in E_{k+1}$
and 
$\rd_{u} f \circ \rd_{f(u)} g = \id_{F_k}$ for $f(u)\in F_{k+1}$.
Here $f(u)\in F_{k+1}$ follows by sc-continuity of $f$ from $u\in E_{k+1}$.
\hfill$\square$
\end{remark}

To construct the example in Theorem~\ref{thm:discontinuous}, we work with an abstract model for the sc-Banach space $\E=(H^{3i}(S^1))_{i\in\N_0}$.
For that purpose we start with an infinite dimensional vector space $$
\textstyle
E:=\bigl\{ \sum_{n=1}^N x_n e_n \,\big|\, N\in\N, x_1,\ldots,x_N \in\R \bigr\}
$$
generated by a sequence of formal variables $(e_n)_{n \in \N}$. 
We obtain norms $\|x\|_i:= \sqrt{\la x, x \ra_i}$ on $E$ by defining inner products with $\la e_n, e_m\ra_i := (nm)^{3i}\delta_{n,m}$. 
Then each completion of $E$ in a norm $\|\cdot\|_i$ defines a Banach space $E_i:= \overline{E}^{\|\cdot\|_i}$, and the embeddings $E_{i+1}\subset E_i$ are compact so that $\E := (E_i)_{i\in\N_0}$ is a sc-Banach space. 
(This follows from the compact Sobolev embeddings $H^{3i}(S^1)\hookrightarrow H^{3j}(S^1)$ for $i>j$. Here an explicit sc-isomorphism $E_0\simeq H^0(S^1)$ mapping $E_i$ to $H^{3i}(S^1)$ can be obtained by taking real and imaginary parts of the complex orthogonal basis $(e^{\sqrt{-1}k\theta})_{k\in\N_0}$ of $L^2(S^1)=H^0(S^1)$ and normalizing these real valued functions to obtain a collection of smooth functions $(e_n)_{n\in\N}\subset \cC^\infty(S^1)=\bigcap_{i\in\N_0} H^{3i}(S^1)$ that have inner products $\la e_n, e_m\ra_{H^{3i}} := n^{6i}\delta_{n,m}$. 
Thus they form an orthonormal basis of $H^0(S^1)$ and the $\|\cdot\|_i$ closure of the finite span $E\hookrightarrow H^0(S^1)$ exactly corresponds to the subspace $H^{3i}(S^1)\subset H^0(S^1)$.)

\begin{proof}[Proof of Theorem~\ref{thm:discontinuous}]
We construct a map $s:\F\to\F$ on $\F:=\R\times\E$ by 
$$ \textstyle
s \,:\; (t, x) \;\mapsto\;  (t, s_t(x)) ,
\qquad 
s_t\bigl(\sum_{n=0}^\infty x_n e_n\bigr) :=  \sum_{n=0}^\infty f_n(t)  x_n e_n 
$$
for a sequence of smooth functions $f_n:\R\to [\frac 12, 1], t\mapsto f\bigl(\frac 12(n(n+1)t + 1 - n )\bigr)$ obtained by reparameterizing a smooth function $f:\R\to[\frac 12,1]$ chosen with $f|_{(-\infty,\frac{1}{2}]} \equiv 1$, 
$f|_{[1,\infty)} \equiv \tfrac{1}{2}$, and ${\rm supp}f' \subset (\frac 12 , 1)$. 
First note that by construction we have $f_n|_{(-\infty,\frac{1}{n+1}]} \equiv 1$
and $f_n|_{[\frac{1}{n},\infty)} \equiv \tfrac{1}{2}$.
So the family of linear maps $s_t$ restricts to $s_t= \id_\E$ for $t\leq 0$ and $s_t|_{\E_N}=\frac 12 \id_{\E_N}$ on $\E_N:={\rm span}\{e_n | n\geq N\}$ for $t\geq \frac 1N$.
Thus, $\rd s: \R\times E_{i+1} \to \cL(\R \times E_i, \R\times E_i)$ cannot be continuous for any $i\in\N_0$ since $\rd s (t,x)|_{\{0\}\times E_i}: (0, X) \mapsto (0, s_t(X))$ is discontinuous at $t=0$ in $\cL(E_i, E_i)$ by 
$$
\|s_{1/n}- s_0\|_{\cL(E_i, E_i)} \geq
\|s_{1/n}(e_n)- s_0(e_n)\|_i \|e_n\|_i^{-1}
= \|\tfrac 12 e_n - e_n \|_i \|e_n\|_i^{-1} = \tfrac 12 .
$$
On the other hand, since $f_n(t)\neq 0$, the map $s$ has an evident inverse given by
$$ \textstyle
s^{-1} \,: \; \bigl(t, \sum_{n=0}^\infty y_n e_n\bigr) \;\mapsto\; \bigl(t, \sum_{n=0}^\infty \tfrac{y_n}{f_n(t)} e_n \bigr) .
$$
To prove the theorem it remains to show that $s$ and $s^{-1}$ are well-defined and sc-smooth. For that purpose note that $s^{-1}$ is of the same form as $s$, with the function $f$ replaced by $\frac 1{f}$. 
So it suffices to consider the map $s$, as long as we only use common properties of the functions $f_n$ in both cases. 
Since ${\rm supp}f_1' \subset (\frac 12 , 1)$ and
the derivatives of $f_1=f$ and $f_1=f^{-1}$ are uniformly bounded, we have for all $n\in\N$
\begin{equation}\label{fest}
\supp f_n^{(k)} \subset \bigl(\tfrac{1}{n+1},\tfrac{1}{n}\bigr) \quad\forall k\geq 1,
\qquad
\bigl\|f^{(k)}_n\bigr\|_\infty = \bigl(\tfrac{n(n+1)}{2}\bigr)^k \bigl\|f^{(k)}_1\bigr\|_\infty \leq n^{2k} C_k \quad\forall k\geq 0.
\end{equation}
Next, we write $s(t,x)=(t,\rho_0(t,x))$ and -- to prove that $\rho_0:\R\times\E\to\E$ and thus $s$ is well-defined and sc-smooth -- we more generally study the maps arising from the derivatives $f^{(k)}_n=\frac{\rd^k}{\rd t^k} f_n$ on shifted sc-spaces $\E^k:=(E_{k+i})_{i\in\N_0}$ for $k\in\N_0$,
$$\textstyle
\rho_k \,:\; \R \times \E^k \;\to\; \E, 
\qquad 
\bigl(t,\sum_{n=0}^\infty x_n e_n\bigr) \;\mapsto\; \sum_{n=0}^\infty f^{(k)}_n(t) x_n e_n .
$$
We can rewrite this
$\rho_k(t,\cdot) = \sum_{n=0}^\infty f^{(k)}_n(t) p_n$ in terms 
of the orthogonal projections to $\R e_n \subset E_0$, 
$$ \textstyle
 p_n \,:\; \E\;\to\;\E, \qquad x \;\mapsto\; \langle x , e_n \rangle_0 \, e_n .
$$
Then for $k\geq 1$ the supports of $f^{(k)}_n$ are disjoint, so we have
$\rho_k (t, \cdot ) = f^{(k)}_{N_t}(t) p_{N_t}$ with $N_t:=\lfloor t^{-1} \rfloor$ for $t>0$ and $\rho_k(t, \cdot)\equiv 0$ for $t\leq 0$ as well as in a small neighborhood $t\sim \frac 1n$ for each $n\in\N$.
Note also for future purposes the estimates for $x\in E_{i+k}$ and $k\geq 0$,
\begin{align}
&\| p_n(x) \|_i  \label{pnik}
\;=\; \bigl| \la x , e_n \ra_0 \bigr| \tfrac{\| e_n \|_i}{\| e_n \|_{i+k}} \|e_n\|_{i+k} 
\;=\; n^{-3k}  \|  \la x , e_n \ra_0 e_n\|_{i+k}
\;=\; n^{-3k} \| p_n(x) \|_{i+k} , \\
&\textstyle \label{pnsum}
\bigl\| \sum_{n=N}^\infty p_n(x) \bigr\|_i
\;=\; \bigl( \sum_{n=N}^\infty \| p_n(x) \|_i^2 \bigr)^{1/2}
\;=\; \bigl( \sum_{n=N}^\infty  n^{-6k} \| p_n(x) \|_{i+k}^2 \bigr)^{1/2}\\
&\qquad\;\;\; \textstyle
\;\leq\;  N^{-3k} \bigl( \sum_{n=0}^\infty  \| p_n(x) \|_{i+k}^2 \bigr)^{1/2}
\;=\;  N^{-3k} \bigl\| \sum_{n=0}^\infty p_n(x) \bigr\|_{i+k}
\;=\;  N^{-3k} \| x\|_{i+k}.
\nonumber
\end{align}
We will show for all $k\in\N_0$ that $\rho_k: \R \times \E^k \to \E$ is well-defined, sc$^0$, and sc-differentiable with tangent map $\rT\rho_k=(\rho_k, \rD\rho_k): \R\times \E^{k+1} \times \R \times \E^k \to \E^1 \times \E$ given by
\begin{equation}\label{eqn:tan}
\rD\rho_k \,:\; (t,x,T,X) \;\mapsto\; \rho_k(t,X) + T \cdot  \rho_{k+1}(t,x) .
\end{equation}
Once this is established, $\rT\rho_k$ is sc$^0$ by scale-continuity of $\rho_k,\rho_{k+1}$. In fact, $\rT\rho_k$, as a sum and product of sc$^1$ maps, is sc$^1$, and further induction proves that $\rho_k$ and thus also $s$ and $s^{-1}$ are all sc$^\infty$.

The above claims and \eqref{eqn:tan} for $t\neq 0$ follow from the maps $\rho_k: E_{k+i}\to E_i$ all being classically differentiable with differential
\begin{align*}
\rD \rho_k(t,x,T,X) 
&\;=\;
\tfrac{\rd}{\rd s}\big|_{s=0} \rho_k(t+s T,x+s X)
\;=\; \textstyle
\tfrac{\rd}{\rd s}\big|_{s=0} 
\sum_{n=0}^\infty f^{(k)}_n(t+s T) p_n(x+s X) \\
& \;=\;  \textstyle
\sum_{n=0}^\infty \bigl( T f^{(k+1)}_n(t) p_n(x) + f^{(k)}_n(t) p_n(X) \bigr)
\;=\; T \cdot \rho_{k+1}(t,x) + \rho_k(t,X).
\end{align*}
To see that $\rho_0$ is well-defined note that $(e_n)_{n\in\N_0}\subset E_i$ is orthogonal on each scale $i\in\N_0$, so
\begin{align*}\textstyle
\bigl\|\rho_0 (t, x)\bigr\|_i 
&\;=\; 
\textstyle
\bigl\| \sum f_n(t) p_n(x) \bigr\|_i
\;=\; \bigl( \sum f_n(t)^2 \|  p_n(x) \|_i^2 \bigr)^{1/2}
\;\leq\; \bigl( \sup_n \|f_n\|_\infty^2 \sum \|p_n(x) \|_i^2 \bigr)^{1/2} \\
&\;=\; \textstyle
\sup_n \|f_n\|_\infty \cdot \bigl\|\sum p_n(x) \bigr\|_i
\;=\; \|f_1\|_\infty \|x\|_i  \;\leq\; 2 \|x\|_i , 
\end{align*}
where $\|f_1\|_\infty = \|f\|_\infty= 1$ or $\|f_1\|_\infty= \|\tfrac 1f\|_\infty = 2$ if we choose $f:\R\to\R$ with values in $[\frac 12, 1]$.

To check sc-continuity of $\rho_0$ at $t=0$ we fix a level $i\in\N_0$ and $x\in E_i$ and estimate for $\R\times E_i\ni (t,h)\to 0$ with $N_t:=\lfloor t^{-1} \rfloor$ for $t>0$ and $N_t:=\infty$ for $t\leq 0$
\begin{align*} 
\|\rho_0(t,x+h) - \rho_0(0,x)\|_i
&\;=\;
\|\rho_0(t, h) + \rho_0(t,x) - x\|_i  
\;\le\; \textstyle
\|\rho_0(t, h)\|_i + \|\sum (f_n(t)-1) p_n(x) \|_i  \\
& \;\le\; \textstyle
2 \|h\|_i + \bigl\| \sum_{n=N_t}^\infty (f_n(t)-1) p_n(x) \|_i \\
&\;\le\; \textstyle
2 \|h\|_i + \sup_n \|f_n -1 \|_\infty \bigl\| \sum_{N_t}^\infty p_n(x) \bigr\|_i 
\;\underset{|t|+\|h\|_i \to 0}{\longrightarrow}\; 0 . 
\end{align*}
Here we used the facts that $f_n(t)=1$ for $n\leq t^{-1}-1$, and that $x=\lim_{N\to\infty} \sum_{n=0}^N p_n(x)\in E_i$ converges, hence as $N_t=\lfloor t^{-1}\rfloor \to \infty$ with $t\to 0$ we have
$\bigl\| \sum_{n=N_t}^\infty p_n(x) \bigr\|_i  \to 0$.

Differentiability of $\rho_0$ with $D\rho_0(0,x,T,X)=\rho_0(0,X)+T\rho_1(0,x) = X$ as claimed in \eqref{eqn:tan} amounts to estimating for $x\in E_{i+1}$ and $t>0$, using \eqref{fest} and \eqref{pnsum},
\begin{align*}
&\bigl\|\rho_0(t, x+X) - \rho_0(0, x) - \rho_0(0,X) \bigr\|_i \\
&\qquad\qquad=\; \textstyle 
\bigl\|\sum f_n(t) p_n(x+X)  - x - X \bigr\|_i
\;=\; \textstyle
\bigl\|\sum_{n=N_t}^\infty (f_n(t)-1) p_n(x + X) \bigr\|_i  \\
&\qquad\qquad\leq\; \textstyle
 \sup_n \|f_n-1\|_\infty \| \sum_{n=N_t}^\infty p_n(x + X)\|_i 
\;\leq\; \textstyle
 N_t^{-3} \| x + X \|_{i+1} ,
\end{align*}
whereas for $t\leq 0$ we have
$\bigl\|\rho_0(t, x+X) - \rho_0(0, x) - \rho_0(0,X) \bigr\|_i= \bigl\| x+X - x - X \|_i =0$. So together we obtain the required convergence of difference quotients,
$$
\frac{\|\rho_0(t, x+X) - \rho_0(0, x) - \rho_0(0,X) \|_i}{|t|+\|X\|_{i+1}} 
\;\leq\; 
\frac{ \max\bigl(0,\lfloor t^{-1} \rfloor^{-3} \bigr) \|x+ X \|_{i+1}}
{|t|+\|X\|_{i+1}}
\underset{|t|+\|X\|_{i+1}\to 0}{\longrightarrow} 0 .
$$
For $k\geq 1$ recall that $\rho_k (t, \cdot ) = f^{(k)}_{N_t}(t) p_{N_t}$ with $N_t=\lfloor t^{-1} \rfloor$ for $t>0$ and $\rho_k(t, \cdot)\equiv 0$ for $t\leq 0$ as well as in a small neighborhood $t\sim \frac 1n$ for each $n\in\N$.
Thus the maps $\rho_k (t, \cdot)$ are evidently well-defined and linear on each scale in $E_i$, and continuous (in fact classically smooth) with respect to $t\in \R\setminus\{0\}$.
To check continuity at $t=0$ we fix a level $i\in\N_0$ and $x\in E_{k+i}$ and estimate for $h\in E_{k+i}$ and $t>0$ 
\begin{align*} 
\|\rho_k(t,x+h) \|_i 
&\;=\;  
\bigl\| f^{(k)}_{N_t}(t) \, p_{N_t}(x+h) \bigr\|_i 
\;\leq\; \| f^{(k)}_{N_t}\|_\infty \| p_{N_t}(x+h) \|_i \\
& \;\leq\; N_t^{2k} C_k  N_t^{-3k}\| x+h \|_{i+k}
\;\leq\; N_t^{-k} C_k \bigl( \|x+h\|_{k+i} \bigr),
\end{align*}
where we used \eqref{fest}, \eqref{pnik}. 
Since $\rho_k(t,x)=0$ for $t\leq 0$ this proves continuity
$$
\|\rho_k(t,x+h) - \rho_k(0,x) \|_i 
\;\leq\; \max\bigl(0,\lfloor t^{-1} \rfloor^{-k} \bigr) C_k 
\bigl( \|x + h\|_{k+i} \bigr)
\;\underset{|t|+\|h\|_{k+i} \to 0}{\longrightarrow}\; 0 . 
$$
Finally, differentiability for $k\geq 1$ with $D\rho_k(0,x,T,X)=\rho_k(0,X)+T\rho_{k+1}(0,x) = 0$ 
as claimed in \eqref{eqn:tan} follows from the analogous estimate for $x\in E_{k+i+1}$ and $t>0$ 
$$
\bigl\|\rho_k(t, x+X) - \rho_k(0, x) - \rho_k(0,X) \bigr\|_i 
\;=\; \bigl\|\rho_k(t, x+X) \bigr\|_i
\;\leq\; N_t^{-k-3} C_k  \|x+X\|_{k+i+1} , 
$$
while for $t\leq 0$ we have
$\bigl\|\rho_k(t, x+X) - \rho_k(0, x) - \rho_k(0,X) \bigr\|_i =0$. So together we obtain the required convergence of difference quotients,
$$
\frac{\|\rho_k(t, x+X) - \rho_k(0, x) - \rho_k(0,X) \|_i}{|t|+\|X\|_{k+i+1}} 
\leq
\frac{ \max\bigl(0,\lfloor t^{-1} \rfloor^{-k-3}\bigr) C_k \|x+ X \|_{k+i+1}}
{|t|+\|X\|_{k+i+1}}
\underset{|t|+\|X\|_{k+i+1}\to 0}{\longrightarrow} 0 .
$$
This proves for all $k\in\N_0$ that $\rho_k$ is sc$^0$ and sc-differentiable with \eqref{eqn:tan}, and thus finishes the proof of sc-smoothness of $s$ and $s^{-1}$.
\end{proof}

\bibliographystyle{amsplain}
\bibliography{references}
\end{document}